\documentclass[11pt]{amsart}
\usepackage{amssymb,latexsym}
\newtheorem{theorem}{Theorem}[section]

\newtheorem{prop}[theorem]{Proposition}
\newtheorem{lemma}[theorem]{Lemma}
\newtheorem{remark}[theorem]{Remark}
\newtheorem{question}[theorem]{Question}
\newtheorem{definition}[theorem]{Definition}
\newtheorem{cor}[theorem]{Corollary}
\newtheorem{example}[theorem]{Example}
\newtheorem{conj}[theorem]{Conjecture}

\begin{document}

\title[Symplectic cones and cohomological properties]{Comparing tamed and compatible symplectic cones and cohomological properties of almost complex manifolds}
\author{Tian-Jun Li }
\address{School  of Mathematics\\  University of Minnesota\\ Minneapolis, MN 55455}
\email{tjli@math.umn.edu}
\author{Weiyi Zhang}
\address{School  of Mathematics\\  University of Minnesota\\ Minneapolis, MN 55455}
\email{zhang393@math.umn.edu}

\begin{abstract}  We  introduce certain homology and cohomology
subgroups for any almost complex structure and study their pureness,
fullness and duality properties. Motivated by a question of
Donaldson, we use these groups to relate $J-$tamed symplectic cones
and $J-$compatible symplectic cones over a large class of  almost
complex manifolds, including all  K\"ahler manifolds, almost
K\"ahler  $4-$manifolds and  complex surfaces.
\end{abstract}

\maketitle

\section{Introduction}

Let $M$ be a closed oriented smooth $2n-$manifold. A symplectic form
compatible with the orientation is a closed 2-form $\omega$ such
that $\omega^n$ is a volume form compatible with the given
orientation. Let $\Omega_M$ be the space of such 2-forms. By taking
the cohomology class, we have the projection map $cc: \Omega_M
\longrightarrow H^2(M;\mathbb R)$. The image $cc(\Omega_M) \subset
H^2(M)$ is called the symplectic cone of $M$, and is denoted by
$\mathcal{C}_M$. In recent years, $\mathcal{C}_M$ has been
extensively studied, especially in dimension 4. In \cite{M}, McDuff
determined $\mathcal{C}_M$ for all the ruled surfaces. A complete
description for the symplectic cone of all the 4-manifolds with
$b^+=1$ was subsequently given in \cite{LL}. After that, several
attempts in different directions were made, $e.g.$ \cite{LU} and
\cite{DL}.

In this note, instead of only fixing an orientation, we further fix a
specific almost complex structure $J$ compatible with the given
orientation. That is to say, we consider an almost complex
$2n-$manifold $(M, J)$. Correspondingly, we study the following
subcones of $\mathcal{C}_M$ associated to the almost complex
structure $J$.

\begin{definition} The $J-$tamed symplectic cone is
$$\mathcal K_J^t=\{[\omega]\in H^2(M;\mathbb R)|\hbox{$\omega$ is tamed by
$J$}\},$$ and  the $J-$compatible symplectic cone is
$$\mathcal K_J^c=\{[\omega]\in H^2(M;\mathbb R)|\hbox{$\omega$ is
compatible with $J$}\}.$$
\end{definition}

Recall that $J$  is an automorphism of the tangent bundle $TM$
satisfying $J^2=-$id,  and $J$ is said to tame $\omega$ if $\omega$
is positive on any $J-$line span$(v, Jv)$ where $v$ is a nonzero
tangent vector.  An $\omega-$tamed $J$ is said to be compatible with
$J$ if, further, $\omega$ is $J-$invariant.
 Since $J-$tameness is an open condition and is preserved
under convex combinations,  $\mathcal K_J^t$ is an open convex cone
in $H^2(M;\mathbb R)$. Moreover, cohomologous $J-$tamed forms are
isotopic.  Since $J-$compatibility is also preserved under convex
combinations, $\mathcal K_J^c$ is a convex subcone of $\mathcal
K_J^t$.

There is a notion of integrability of $J$ given by the vanishing of
the Nijenhuis tensor.
 The deep Newlander-Nirenberg Theorem asserts that $J$ is
integrable if and only it is induced by a complex structure. For a
complex structure  $J$, there is the Dolbeault complex and the
associated Dolbeault groups $H^{p,q}_{\bar \partial}(M)$. The
$J-$compatible cone $\mathcal K_J^c$ for a complex structure is just
the usual K\"ahler cone. We call a complex structure $J$ K\"ahler if
its K\"ahler cone is non-empty. In this case, the K\"ahler cone is
an open convex cone of  $H^{1,1}_{\bar
\partial}(M)_{\mathbb R}$, where
$$H^{p,q}_{\bar
\partial}(M)_{\mathbb R}=H^{p,q}_{\bar \partial}(M)\cap H^{p+q}(M;\mathbb
R).$$

More generally, we call $J$ almost K\"ahler if its compatible cone
is not empty. Almost K\"ahler manifolds have been studied from the
point of view of Riemannian geometry ever since Gray. The emphasis
is often on a fixed compatible metric $g$ and its curvature
properties, see e.g. \cite{AD}. Our point of view is different as we
do not fix a priori a metric.

Part of our motivation for studying these cones is due to a question
raised by Donaldson in \cite{D},

\begin{question}\label{D}
If $J$ is an almost complex structure on a compact 4--manifold which
is tamed by a symplectic form, is there a symplectic form compatible
with $J$ (and in the same cohomology class if $b^+=1$)?
\end{question}

Deep results on this question have been made in \cite{W} and
\cite{TWY}. In our language this question is, if $\mathcal K_J^t$ is
not empty for some $J$, then is $\mathcal K_J^c$ nonempty as well?
This question only makes sense in dimension 4, because, for $n\geq
3$, $\mathcal K_J^c$ is empty even locally for a generic almost
complex structure (c.f. \cite{Le} and \cite{Tom}).

Also motivated by  Question \ref{D}, we introduce in section 2 the
analogues of the (real) Dolbeault groups
$$H^{1,1}_{\bar
\partial}(M)_{\mathbb R}\quad \hbox{and} \quad   (H^{2,0}_{\bar
\partial}(M)\oplus H^{0,2}_{\bar
\partial}(M))_{\mathbb R}$$
for general $J$, which we denote by
\begin{equation}\label{2} H^{1,1}_{J}(M)_{\mathbb R}\quad \hbox{and} \quad
H^{(2,0),(0,2)}_{J}(M)_{\mathbb R}\end{equation} respectively (see
Definition \ref{new}).

An important observation related to Question \ref{D} is, as in the
integrable case, the $J-$compatible cone is an open (possibly empty)
convex cone of $H^{1,1}_{J}(M)_{\mathbb R}$.

  We analyze for what $J$  the groups in \eqref{2} give rise to a direct sum
  decomposition of $H^2(M;\mathbb R)$. Such a $J$ is called
  $C^{\infty}$ pure and full.

We also introduce the corresponding groups for currents
\begin{equation}\label{2'} H_{1,1}^{J}(M)_{\mathbb R}\quad \hbox{and} \quad
H_{(2,0),(0,2)}^{J}(M)_{\mathbb R}\end{equation} and introduce the notion of pure and full when the groups in \eqref{2'} give rise to a direct sum
  decomposition of $H_2(M;\mathbb R)$.

Using ideas in \cite{S} and \cite{HL} we prove in section 3

\begin{theorem}\label{extension}
 Suppose $J$ is a  $C^{\infty}$  full almost complex structure.
 If $\mathcal K_J^c$ is non-empty
 then
$$\mathcal K_J^t=\mathcal
K_J^c+H_J^{(2,0),(0,2)}(M)_{\mathbb R}.$$
\end{theorem}

We want to point out that Theorem \ref{extension} applies to
K\"ahler complex structures. And in this case, it is not hard to see
that  $H_J^{(2,0),(0,2)}(M)_{\mathbb R}$ is isomorphic to $(H_{\bar
\partial}^{2,0}(M)\oplus H_{\bar \partial}^{0,2}(M))_{\mathbb R}$ (\cite{DLZ}).

For general almost complex structures the cohomology subgroups in
\eqref{2} and their homology analogues in \eqref{2'} seem to have
not been systematically explored in the literature. We believe that
they are important invariants of almost complex structures and
deserve further study. We would like to mention that there are two
recent papers \cite{DLZ} and \cite{FT} which are closely related to
this work (see Remark \ref{last}). In particular, it is shown in
\cite{DLZ} that any $4-$dimensional almost complex structure is
$C^{\infty}$ pure and full. It is easy to observe that
$H_J^{(2,0),(0,2)}(M)_{\mathbb R}$ is always trivial if $b^+(M)=1$.
Together with Theorem \ref{extension}, we obtain,

\begin{cor}\label{dlz}
 Suppose $(M, J)$ is an almost complex $4-$manifold
 with non-empty $\mathcal K_J^c$.
 Then
$$\mathcal K_J^t=\mathcal
K_J^c+H_J^{(2,0),(0,2)}(M)_{\mathbb R}.$$ In particular, if
$b^+(M)=1$, then $K_J^t=K_J^c$.
\end{cor}

This provides some positive evidence for question \ref{D},
especially in the case of $b^+=1$.

When $b^+>1$, the calculations in \cite{DLZ} also demonstrate that
$H_J^{(2,0),(0,2)}(M)_{\mathbb R}$ often vanishes. Thus, in
light of Corollary \ref{dlz}, we speculate that we might actually
have the  equality $\mathcal K_J^t=\mathcal K_J^c$ for a generic
$4-$dimensional almost complex structure.

In section 4 we focus on complex structures. We observe in this case
Theorem \ref{extension} is also a direct consequence of the deep
Nakai-Moishezon type K\"ahler criterion of \cite{DP}\footnote{and
\cite{B}, \cite{L} when $n=2$}.  We also point out the parallel to
some classical results in algebraic geometry and K\"ahler geometry.
  Further, if we let $H\mathfrak
C(J)$ be the cone of homology complex cycles, in the sense of
Sullivan \cite{S}, then we can describe it using the $J$-compatible
symplectic cone and the analytic subsets of $M$.

Moreover, for complex surfaces, we confirm Question \ref{D}.

\begin{theorem} \label{main} Let $J$ be a  complex structure on a $4-$manifold $M$.
 Then $\mathcal K_J^t(M)$ is empty if and
only if $\mathcal K_J^c(M)$ is
 empty.
\end{theorem}

It is
 a direct consequence of several remarkable results in complex surface
theory: the Kodaira classification \cite{BPV}, the K\"ahler
criterion of $b^+$ being odd (\cite{Siu}, \cite{T}, \cite{Mi},
\cite{B}), and the analysis of complex curves in non-K\"ahler
elliptic surfaces (\cite{HL}).

 Finally we compare the
union of tamed cones and the union of compatible cones over all
complex structures in dimension 4.

The authors wish to thank S-T. Yau for stimulating questions. We
also appreciate  F. Zheng, S. Morgan, V. Tosatti, B. Weinkove  for
discussions, S. Donaldson, M. Furuta, R. Hind, K. Ono for their
interest, A. Fino, A. Tomassini for sending to us their preprint
\cite{FT}. Finally we are grateful to T. Dr$\breve{a}$ghici for his
interest and pointing out the reference \cite{Dr2}.
The  authors are partially
supported by NSF and the McKnight foundation.

\tableofcontents

\section{Some homology and cohomology subgroups of almost complex
manifolds}
We study in this section decompositions of forms and currents on
almost complex manifolds and the associated cohomology and homology
subgroups.

\subsection{Forms and currents}

On a smooth closed manifold $M$, the  space $\Omega^*(M)$ of
$C^{\infty}$ form is a vector space, and  with $C^{\infty}$
topology, it is a Fr\'echet space, i.e. a complete metrizable
locally convex topological vector space. The space $\mathcal E_*(M)$
of currents is the topological dual space, which is also a Fr\'echet
space. As a topological vector space, $\Omega^*(M)$ is reflexive,
thus it is also the dual space of $\mathcal E_*(M)$.

The exterior derivative on $\Omega^*(M)$ induces a boundary operator
on $\mathcal E_*(M)$, making it also into a complex. By abusing
notation we denote both  the differentials in the current complex
and the form complex by $d$.

Denote the space of closed forms by $\mathbf Z$ and the space of
exact forms by $\mathbf B$. Denote the space of closed currents by
$\mathcal Z$ and the space of boundaries by $\mathcal B$. $\mathbf
Z$ and $\mathcal Z$  are  closed subspaces since $d$ is continuous.
 It is easy to check  a current is closed if and
only if it vanishes on  $\mathbf B$, and a form is closed if and
only if it vanishes on space $\mathcal B$.

We call the homology groups of the complex of  currents the De Rham
homology groups.
 The inclusion
of smooth forms into the currents induces a natural isomorphism of
the $(2n-k)-$th De Rham cohomology group and the $k-$th De Rham
homology group.  Thus each closed  $k-$current is homologous to a
smooth $(2n-k)-$form. Moreover, by Theorem 17' in \cite{DR}, a
current is a boundary if and only if it vanishes on $\mathbf Z$, and
a smooth form is a boundary if and only if it vanishes on $\mathcal
Z$ (see \cite{DR}). This implies in particular that both $\mathbf B$
and $\mathcal B$ are also closed subspaces.

Let $J$ be an almost complex structure on a smooth manifold $M$.
Then for each non-negative  integer $k$, there is a natural action
of $J$ on the space
 $\Omega^k(M)_\mathbb{C}$ of complex smooth differential $k-$forms
${\Omega^k(M)}_{\mathbb C}:=\Omega^k(M)\otimes \mathbb C$, which
induces
   a topological type
decomposition
$${\Omega^k(M)}_{\mathbb C}=\oplus_{p+q=k} \Omega_J^{p,q}(M)_{\mathbb C}.$$
If $k$ is even, $J$ also acts on $\Omega^k(M)$ as an involution and
decomposes it into the topological direct sum of the invariant part
$\Omega_J^{2, +}(M)$ and the anti-invariant part
$\Omega_J^{2,-}(M)$. We are particularly interested in the case
$k=2$. In this case, the two decompositions are related in the
following way:
$$\begin{array}{llll}\Omega^{2,+}_J(M)&=\Omega_J^{1,1}(M)_{\mathbb R}&:=&\Omega_J^{1,1}(M)_{\mathbb C}\cap
\Omega^2(M),\cr \Omega^{2,-}_J(M)&=\Omega_J^{(2,0),
(0,2)}(M)_{\mathbb R}&:=&(\Omega_J^{(2,0)}(M)_{\mathbb C}\oplus
\Omega_J^{0,2}(M)_{\mathbb C})\cap \Omega^2(M).
\end{array}
$$
 Let $\mathbf Z_J^{1,1}\subset \Omega^{1,1}_J(M)_{\mathbb R}$ be the subspace of real closed
 $(1,1)$ forms, and $\mathbf B_J^{1,1}\subset \Omega^{1,1}_J(M)$ be the subspace
of real exact $(1,1)$ forms. Similarly define subspaces of
$\Omega_J^{(2,0), (0,2)}(M)_{\mathbb R}$, $\mathbf Z_J^{(2,0),
(0,2)}$ and $\mathbf B_J^{(2,0), (0,2)}$.

 For the space of real $2-$currents, we have a similar decomposition,
$$\mathcal E_2(M)=\mathcal E^J_{1,1}(M)_{\mathbb R}\oplus \mathcal
E^J_{(2,0), (0,2)}(M)_{\mathbb R},$$ and the corresponding subspaces
of closed, and boundary currents,
 $$\begin{array}{lllll}\mathcal B^J_{1,1}&\subset & \mathcal Z^J_{1,1}&\subset &\mathcal E^J_{1,1}(M)_{\mathbb
 R},\cr
\mathcal B^J_{(2,0), (0,2)}&\subset&  \mathcal Z^J_{(2,0),
(0,2)}&\subset &\mathcal E^J_{(2,0), (0,2)}(M)_{\mathbb
 R}.
 \end{array}
 $$

We note the dual space of $\mathcal E^J_{1,1}(M)_{\mathbb R}$ is
$\Omega_J^{1,1}(M)_{\mathbb R}$, and vice versa. The same can be
said for  $\mathcal E^J_{(2,0), (0,2)}(M)_{\mathbb R}$ is
$\Omega_J^{(2,0), (0,2)}(M)_{\mathbb R}$
 Since closed subspaces of a Fr\'echet space is Fr\'echet,
all the spaces introduced above are  Fr\'echet spaces.

\subsection{Pureness and fullness}
\subsubsection{$C^{\infty}$ pure and full almost complex structures}

\begin{definition}\label{new}  If $S=(1,1)$, or $(2,0), (0,2)$, define
$$H_J^{S}(M)_\mathbb{R}=\{[\alpha]\in H^*(M;\mathbb R)
|\alpha\in \mathbf Z_J^{S}\}=\frac{\mathbf Z_J^{S}}{\mathbf
B_J^{S}}.$$

 \end{definition}
It is simple to observe

\begin{lemma}\label{opensubset} The $J-$compatible cone $\mathcal K_J^c$ is contained in
$ H^{1,1}_J(M)_{\mathbb R}$ as an open subset.
\end{lemma}
 Thus it is crucial to understand $ H^{1,1}_J(M)_{\mathbb R}$.
Notice that a $J-$tamed symplectic form is a closed smooth $2-$form
whose $(1,1)$ component is positive definite, but not necessarily
closed.

\begin{remark} \label{notation} Notice that $\frac{\mathbf Z_J^{S}}{\mathbf B}=
\frac{\mathbf Z_J^{S}}{\mathbf B\cap \mathbf Z^J_{S} }=\frac{\mathbf
Z_J^{S}}{\mathbf B_J^{S}}$. Thus there is a natural inclusion
$$\rho_{S}:\frac{\mathbf
Z_J^{S}}{\mathbf B_J^S}\longrightarrow \frac{\mathbf Z}{\mathbf
B},$$ and $$H_J^S(M)_\mathbb{R}=\{[\alpha]\in H_*(M;\mathbb
R)|\alpha\in \mathbf Z_J^{S}\}.$$
 By abusing notation we will  write $\rho_S(\frac{\mathbf
Z_J^S}{\mathbf B_J^S})\subset \frac{\mathbf Z}{\mathbf B}$ simply as
$\frac{\mathbf Z_J^S}{\mathbf B_J^S}$. If there is no confusion we
will also write $\Omega^S$ for $\Omega_J^S$, $\mathbf Z^S$ for
$\mathbf Z_J^S$ and $\mathbf B^S$ for $\mathbf B_J^S$.
\end{remark}

With the notation convention in the above remark, in particular,  we
have
$$\frac{\mathbf Z^2}{\mathbf B^2} \supset \frac{\mathbf
Z^{1,1}}{\mathbf B^{1,1}}+ \frac{\mathbf Z^{(2,0), (0,2)}}{\mathbf
B^{(2,0), (0,2)}},$$ i.e. $H^2(M;\mathbb R)\supset
H_J^{1,1}(M)_\mathbb{R}+H_J^{(2,0), (0,2)}(M)_\mathbb{R}$.

In this subsection we study when the type decomposition holds for
$H^2(M;\mathbb{R})$, i.e.
$$H^2(M;\mathbb R)=H_J^{1,1}(M)_\mathbb{R}\oplus H_J^{(2,0), (0,2)}(M)_\mathbb{R}.$$

\begin{definition}\label{p}
 $J$ is said to be $C^{\infty}$  pure
 if $$\frac{\mathbf Z^{1,1}}{\mathbf
B^{1,1}}\cap \frac{\mathbf Z^{(2,0), (0,2)}}{\mathbf B^{(2,0),
(0,2)}}=0.$$
\end{definition}

\begin{definition}\label{f}
 $J$ is said to be $C^{\infty}$ full if
$$\frac{\mathbf Z}{\mathbf B} = \frac{\mathbf Z^{1,1}}{\mathbf B^{1,1}}+
\frac{\mathbf Z^{(2,0), (0,2)}}{\mathbf B^{(2,0), (0,2)}}.$$

\end{definition}

It is immediate from Definition \ref{p} and \ref{f}  that we have

\begin{lemma}\label{fptd}
 $J$ is  $C^{\infty}$ pure and full if and only if  we have the type decomposition
\begin{equation}\label{type3}
\begin{array}{ll}
 H^2(M;\mathbb R)&=H_J^{1,1}(M)_{\mathbb R}\oplus
H_J^{(2,0), (2,0)}(M)_{\mathbb R}.
\end{array}
\end{equation}
\end{lemma}

Let $\pi^{S}:\Omega ^2(M)\to \Omega^{S}(M)_{\mathbb R}$ be the
natural projection. Notice that $\mathbf B^{S}=d\Omega^1(M)\cap
\Omega^{S}(M)_{\mathbb R}$ is a proper subspace of $\pi^{S}\mathbf
 B$.
 In particular,  $\pi^{1,1}\mathbf B$ is
the subspace of  $(1,1)-$forms which are
 components of exact forms and $\pi^{1,1}\mathbf Z$ is the subspace of  $(1,1)-$forms
 which are
 components of closed forms.

It is important to understand the quotient spaces
$$\frac{\pi^{1,1}\mathbf Z}{\pi^{1,1}\mathbf B},\quad
\frac{\pi^{(2,0), (0,2)}\mathbf Z}{\pi^{(2,0), (0,2)}\mathbf B}.$$
 Since $\mathbf
Z^{1,1}\subset \pi^{1,1}\mathbf Z$ and $\mathbf B^{1,1}\subset
\pi^{1,1}\mathbf B$, there is a natural homomorphism
$$\iota^{1,1}:\frac{\mathbf Z^{1,1}}{\mathbf
B^{1,1}}\longrightarrow \frac{\pi^{1,1}\mathbf Z}{\pi^{1,1}\mathbf
B}.$$ Similarly there is a homomorphism
$\iota^{(2,0),(0,2)}:\frac{\mathbf Z^{(2,0), (0,2)}}{\mathbf
B^{(2,0), (0,2)}}\longrightarrow \frac{\pi^{(2,0), (0,2)}\mathbf
Z}{\pi^{(2,0), (0,2)}\mathbf B}$.

Let $W\subset \mathbf Z^2$ be a subspace lifting $H^2(M;\mathbb R)$.
As mentioned $\mathbf B^2$ is a closed subspace of $\mathbf Z^2$,
and $W$ is also a closed subspace of $\mathbf Z^2$ since it is of
finite dimensional. So we have a direct sum decomposition
$$\mathbf Z^2=\mathbf B^2\oplus W.$$
In turn, it gives rise to decompositions
\begin{equation}\label{sum}\begin{array}{llll} \pi^{1,1}\mathbf Z&=\pi^{1,1}\mathbf
B&+&\pi^{1,1}W,\cr \pi^{(2,0), (0,2)}\mathbf Z&=\pi^{(2,0),
(0,2)}\mathbf B&+&\pi^{(2,0), (0,2)}W.
\end{array}
\end{equation}

Let  $\overline{\pi^{1,1}\mathbf B}$ be the closure of
$\pi^{1,1}\mathbf B$ in $\Omega^{1,1}(M)_{\mathbb R}$. Notice that
$\Omega^{1,1}(M)_{\mathbb R}$ is closed in $\Omega^2(M)$. So
$\overline{\pi^{1,1}\mathbf B}$ is the closure of $\pi^{1,1}\mathbf
B$ in $ \Omega^{2}(M)$. Similarly let $\overline{\pi^{1,1}\mathbf
Z}$ be the closure of $\pi^{1,1}\mathbf Z$ in
$\Omega^{1,1}(M)_{\mathbb R}$.

There is  the natural homomorphism
$$\phi^{1,1}: \frac{\pi^{1,1}\mathbf Z}{\pi^{1,1}\mathbf
B}\to \frac{\overline{\pi^{1,1}\mathbf
Z}}{\overline{\pi^{1,1}\mathbf B}}.$$ Similarly there is the
homomorphism $\phi^{(2,0), (0,2)}$ for $(2,0), (0,2)$ forms.

\begin{lemma}\label{schaefer}
$\phi^{1,1}$ and $\phi^{(2,0), (0,2)}$ are   surjective.
\end{lemma}

\begin{proof}
Notice that $\pi^{1,1}W$ is  a  finite dimensional space. We use the
fact on p22 in \cite{Sch}:

\begin{lemma}\label{tvs} Let $V$ be a topological vector space over $\mathbb R$. If $Q$
is a closed subspace of $V$ and $P$ is a finite dimensional subspace
of $V$, then $Q+P$ is closed in $V$.
\end{lemma}

\begin{proof}Let $t:V\to V/Q$ be the quotient map. Since $Q$ is closed, the
quotient $V/ Q$ is Hausdorff. Thus $t(P)$ is also Hausdorff.
 Since $t(P)$  is also finite dimensional subspace of
$V/ Q$, as a topological space $t(P)$
is isomorphic to $\mathbb R^k$ for some integer $k$ by Theorem 3.2 in \cite{Sch}. Thus $t(P)$ is
complete in the sense that each Cauchy filter converges. Therefore
$t(P)$ is closed in $V/ Q$. (For general $P$, even $P$ is closed in
$V$, it is not necessary that the image $t(P)$ is closed in $V/Q$).
Finally, since $t$ is continuous, $Q+P=t^{-1}(t(P))$ is a closed
subspace of $V$.
\end{proof}

Let $\tilde W^1\subset \pi^{1,1}W$ be a subspace mapping
isomorphically to $\frac{\pi^{1,1}\mathbf Z}{\pi^{1,1}\mathbf B}$,
and similarly define $\tilde W^2\subset \pi^{(2,0), (0,2)}W$. Then
we still have \eqref{sum} with $\tilde W^1$ and $\tilde W^2$.

Apply Lemma \ref{tvs} to the case
$$V=\Omega^{1,1}(M)_{\mathbb R}\,\,\,\,
(\Omega^{(2,0), (0,2)}(M)_{\mathbb R}), \quad
Q=\overline{\pi^{1,1}\mathbf B}\,\,\,\, (\overline{\pi^{(2,0),
(0,2)}\mathbf B}), \quad P=\tilde W^1\,\, (\tilde W^2) $$ to obtain
\begin{equation}\label{sumbar}\begin{array}{llll}
\overline{\pi^{1,1}\mathbf Z}&=\overline{\pi^{1,1}\mathbf
B}&+&\tilde W^1,\cr \overline{\pi^{(2,0), (0,2)}\mathbf
Z}&=\overline{\pi^{(2,0), (0,2)}\mathbf B}&+&\tilde W^2.
\end{array}
\end{equation}
Therefore $\tilde W^1$ ($\tilde W^2)$  projects surjectively onto
$$\frac{\overline{\pi^{1,1}\mathbf
Z}}{\overline{\pi^{1,1}\mathbf B}} \quad
(\frac{\overline{\pi^{(2,0), (0,2)}\mathbf Z}}{\overline{\pi^{(2,0),
(0,2)}\mathbf B}}).$$
\end{proof}

\begin{lemma}\label{45}
$J$ is $C^{\infty}$ pure if and only if
\begin{equation} \label{hl}\pi^{1,1}\mathbf B\cap \mathbf Z^{1,1}=\mathbf
B^{1,1},
\end{equation}
and (\ref{hl}) is equivalent to
\begin{equation} \label{hl2}\pi^{(2,0), (0,2)}\mathbf B\cap \mathbf Z^{(2,0), (0,2)}=\mathbf B^{(2,0),
(0,2)}.
\end{equation}
Consequently,  $J$ being  $C^{\infty}$ pure is equivalent to
$\iota^{1,1}$ being injective, which is also equivalent to
$\iota^{(2,0), (0,2)}$ being injective.
\end{lemma}
\begin{proof} The equivalence between (\ref{hl}) and (\ref{hl2})
follows from that $\pi^{1,1}d\gamma$ is closed (exact) if and only
if $\pi^{(2,0), (0,2)}d\gamma$ is closed (exact).

Suppose (\ref{hl}) is true. We want to prove that $J$ is
$C^{\infty}$ pure, i.e. if $e\in \mathbf Z^{1,1}$, $f\in \mathbf
Z^{(2,0), (0,2)}$ and $[e]=[f]\in H^2(M;\mathbb R)$, then
$[e]=[f]=0$.

Since $e-f\in \mathbf B$, $\pi^{1,1}(e-f)=e\in \pi^{1,1}\mathbf B$.
As $e\in \mathbf Z^{1,1}$, we conclude that $e\in \mathbf B^{1,1}$
and hence $[e]=[f]=0$. This proves that $J$ is pure.

Conversely, suppose $J$ is pure. We need to show that if  $e\in
\mathbf Z^{1,1}$ and $e=\pi^{1,1}d\gamma$, then $[e]=0$.

Let $-f=d\gamma-e$. Then $df=0$ and thus $f\in  \mathbf Z^{(2,0),
(0,2)}$. Since $[e-f]=0$ we conclude that $[e]=0$.
\end{proof}

\begin{remark}
It is not clear we can replace $\pi^{1,1}\mathbf B$  in Lemmas
\ref{45} by its closure $\overline{\pi^{1,1}\mathbf B}$.
\end{remark}

Suppose $J$ is $C^{\infty}$ full. Then we can choose $W=W^1\oplus W^2$ with
  $W^1\subset  \mathbf Z^{1,1}$   a
subspace lifting $H^{1,1}_J(M)_{\mathbb R}$, and
$W^2 \subset \mathbf Z^{(2,0), (0,2)}$  a subspace lifting
$H^{(2,0), (0,2)}_J(M)_{\mathbb R}$.  Then $\pi^{1,1}W=W^1$ and $\pi^{(2,0), (0,2)}W=W^2$.
This proves

\begin{lemma}\label{iota'}
If $J$ is $C^{\infty}$ full then both $\iota^{1,1}$ and $\iota^{(2,0), (0,2)}$
are surjective. Consequently, if  $J$ is $C^{\infty}$ pure and full, then both
$\iota^{1,1}$ and $\iota^{(2,0), (0,2)}$ are isomorphisms.
\end{lemma}

 Notice that when $J$ is integrable, there
is the Dolbeault complex and the associated Dolbeault cohomology
groups. But our groups  are subgroups of the De Rham homology and
cohomology groups, and might be different from the Dolbeault groups.
Nonetheless, the following identification is  made in \cite{DLZ}.

\begin{prop} \label{complex1}
Let $J$ be a complex structure on a $2n-$manifold. If $J$ is
K\"ahler or $n=2$ then $J$ is $C^{\infty}$ pure and full. Moreover,
\begin{equation}\label{iden}\begin{array}{ll}
H_J^{1,1}(M)_{\mathbb R}&=H_{\bar \partial}^{1,1}(M)_{\mathbb R},
\cr H_J^{(2,0), (0,2)}(M)_{\mathbb R}&=(H_{\bar
\partial}^{2,0}(M)\oplus H_{\bar \partial}^{0,2}(M))_{\mathbb R}.
\end{array}\end{equation}
\end{prop}

\begin{remark}\label{last} It is interesting to investigate whether some non-integrable $J$ could
be $C^{\infty}$ pure and full.

It is shown in \cite{DLZ} and \cite{FT} that if $J$ is almost
K\"ahler, then it is $C^{\infty}$ pure. While in \cite{FT} a
left-invariant almost complex structure on a $6-$dimensional
nilmanifold is found to be not $C^{\infty}$ pure.

 In \cite{DLZ}  it is shown   that every almost
complex structure on a $4-$dimensional manifold is $C^{\infty}$ pure
and full. Some higher dimensional examples are also given in
\cite{FT}.

In addition, in dimension $4$, the groups $$H_J^{1,1}(M)_{\mathbb
R}, \quad H_J^{(2,0), (0,2)}(M)_{\mathbb R}$$ are determined in
\cite{DLZ} for any $J$ metric related to an integrable one (i.e.
when there is a metric $g$ which is compatible with $J$ and an
integrable almost complex structure at the same time), and their
dimensions are estimated for those $J$ tamed by symplectic forms.
\end{remark}


\subsubsection{Pure and full almost complex structures}

\begin{definition}\label{new}   If $S=(1,1)$ or $(2,0), (0,2)$,
define
$$H^J_S(M)_\mathbb{R}=\frac{\mathcal Z^J_S}{\mathcal B^J_S}.$$
 \end{definition}

We will adopt similar notation convention as in Remark
\ref{notation}.

\begin{definition}
$J$ is said to be pure if $$\frac{\mathcal Z_{1,1}}{\mathcal
B_{1,1}}\cap \frac{\mathcal Z_{(2,0), (0,2)}}{\mathcal B_{(2,0),
(0,2)}}=0.$$
 $J$ is said to be full if
$$\frac{\mathcal Z_2}{\mathcal B_2} = \frac{\mathcal Z_{1,1}}{\mathcal B_{1,1}}+
\frac{\mathcal Z_{(2,0), (0,2)}}{\mathcal B_{(2,0), (0,2)}}.$$
\end{definition}

Clearly, the analogue of Lemma \ref{fptd} still holds.
\begin{lemma}
 $J$ is  pure and full if and only if  we have the type decomposition
\begin{equation}\label{type}
\begin{array}{ll}
 H_2(M;\mathbb R)&=H^J_{1,1}(M)_{\mathbb R}\oplus
H^J_{(2,0), (2,0)}(M)_{\mathbb R}.
\end{array}
\end{equation}
\end{lemma}

  One can follow
the proof of Proposition \ref{complex1} in \cite{DLZ} and
 work with dual complex and dual operations to prove (see also
 Remark \ref{ft})

\begin{prop}\label{complex3}
Let $J$ be a complex structure on a $2n-$manifold. If $J$ is
K\"ahler or $n=2$, then it  is pure and full.
\end{prop}



 Let $\pi_{S}:\mathcal
E_2(M)\to \mathcal E_{S}(M)_{\mathbb R}$ be the natural projection.
Notice that $\mathcal B_{S}=d\mathcal E_3(M)\cap \mathcal
E_{S}(M)_{\mathbb R}$ is a proper subspace of $\pi_{S}\mathcal
 B$.
 In particular,  $\pi_{1,1}\mathcal B$ is
the subspace of bidimension $(1,1)$ currents which are
 components of boundaries and $\pi_{1,1}\mathcal Z$ is the subspace of bidimension $(1,1)$
currents which are
 components of cycles.


It is also important to understand the quotient spaces
$$\frac{\pi_{1,1}\mathcal Z}{\pi_{1,1}\mathcal B},\quad
\frac{\pi_{(2,0), (0,2)}\mathcal Z}{\pi_{(2,0), (0,2)}\mathcal B}.$$
Consider the  natural homomorphisms
$$\iota_{1,1}:\frac{\mathcal Z_{1,1}}{\mathcal
B_{1,1}}\longrightarrow \frac{\pi_{1,1}\mathcal Z}{\pi_{1,1}\mathcal
B},
\quad \iota_{(2,0),(0,2)}:\frac{\mathcal Z_{(2,0), (0,2)}}{\mathcal
B_{(2,0), (0,2)}}\longrightarrow \frac{\pi_{(2,0), (0,2)}\mathcal
Z}{\pi_{(2,0), (0,2)}\mathcal B}.$$


\begin{lemma}\label{injective}
$J$ being  pure is equivalent to  $\iota_{1,1}$ being injective,
which is also equivalent to  $\iota_{(2,0), (0,2)}$ being injective.

If $J$ is full then both $\iota_{1,1}$ and $\iota_{(2,0), (0,2)}$
are surjective.

 Consequently, if  $J$ is pure and full, then both
$\iota_{1,1}$ and $\iota_{(2,0), (0,2)}$ are isomorphisms.

\end{lemma}

Of course we also have the homomorphisms $\phi_{1,1}$ and
$\phi_{(2,0), (0,2)}$, and

\begin{lemma}\label{also}
$\phi_{1,1}$ and $\phi_{(2,0), (0,2)}$ are surjective.
\end{lemma}


\subsubsection{Closed almost complex structures}

 To compare homological properties of $J$ for the  complexes of  currents
 and forms
we further introduce  the following condition.

\begin{definition} \label{closed} An almost complex structure $J$ is said to be closed
if $\pi_{1,1}\mathcal B$ is a  closed subspace of $\mathcal
E_{1,1}(M)_{\mathbb R}$.


$J$ is said to be $C^{\infty}$ closed if the image of the operator
\begin{equation}\label{image} d:\Omega^{1,1}(M)_{\mathbb R}\to \Omega^3(M)
\end{equation} is a closed subspace of $ \Omega^{3}(M)$.
\end{definition}

Notice that  $\pi_{1,1}\mathcal B$ is the image of the operator
$$\pi_{1,1}d:\mathcal E_3(M)\to \mathcal
E_{1,1}(M)_{\mathbb R}.$$ The adjoint of this operator $\pi_{1,1}d$
can be easily computed to be the operator \eqref{image}. Recall that
the closed range theorem
 says if a linear map between Fr\'echet spaces has closed range then
its adjoint also has closed range. Thus we have

\begin{lemma}
$J$ is closed if and only if $J$ is $C^{\infty}$ closed.
\end{lemma}

By the same argument for Lemma \ref{schaefer}, we have

\begin{lemma} \label{sigma} If $J$ is a closed almost complex structure, then
 $\pi_{1,1}\mathcal Z$ is also a closed subspace of $\mathcal
E_{1,1}(M)_{\mathbb R}$.
\end{lemma}

Notice that if $J$ is a complex structure, then the image of the
operator \eqref{image} actually lies in $\Omega^{(2,1),
(1,2)}(M)_{\mathbb R}:=(\Omega^{2,1}_J(M)_{\mathbb C}\oplus
\Omega^{1,2}_J(M)_{\mathbb C})\cap \Omega^3(M)$. And it was shown in
\cite{HL} that in this case the image is a closed subspace of
$\Omega^{(2,1), (1,2)}(M)_{\mathbb R}$. Denote this image by $V$. We
write
$$V=(\pi^{(2,1), (1,2)})^{-1}(V)\cap (\pi^{(3,0), (0,3)})^{-1}(0), $$
where  $\pi^{(2,1),(1,2)}$ and $\pi^{(3,0), (0,3)}$ are  projection
operators from $\Omega^3(M)$ to $\Omega^{(2,1), (1,2)}(M)_{\mathbb
R}$ and $\Omega^{(3,0), (0,3)}(M)_{\mathbb
R}:=(\Omega^{3,0}_J(M)_{\mathbb C}\oplus \Omega^{0,3}_J(M)_{\mathbb
C})\cap \Omega^3(M)$ respectively. By the continuity of
$\pi^{(2,1),(1,2)}$ and $\pi^{(3,0), (0,3)}$, we see that $V$ is
also closed in $\Omega^3(M)_{\mathbb R}$. Thus we conclude

\begin{prop} \label{closed} Any complex structure is closed (and also $C^{\infty}$ closed).

\end{prop}

It would be interesting to see whether the explicit deformation
$J_{\alpha}$ of a complex structure $J$ by a holomorphic $2-$form
$\alpha$ in \cite{Lee} is closed.

\subsection{Duality}



In this subsection we compare $H_{1,1}^J(M)_\mathbb{R}$ and
$H^{1,1}_J(M)_\mathbb{R}$.  We recall the following well-known fact.

\begin{lemma}\label{closed-exact}
A real $(1,1)-$form is closed if and only if it vanishes on
$\pi_{1,1}\mathcal B$, and a real $(1,1)-$form is exact  if and only
if it vanishes on $\pi_{1,1}\mathcal Z$. The same is true in the
$(2,0), (0,2)$ setting.
\end{lemma}
\begin{proof} Suppose $\alpha$ is a real closed $(1,1)-$form. Then for
any $u\in \mathcal E(M)_{\mathbb R}$, we have
\begin{equation}\label{0}\alpha\cdot \pi_{1,1}(du)=\alpha\cdot du=d\alpha\cdot
u=0,\end{equation}  i.e. $\alpha$ vanishes on $\pi_{1,1}\mathcal B$.
Conversely, if $\alpha$ vanishes on $\pi_{1,1}\mathcal B$, then we
still have \eqref{0}. This  implies that $d\alpha=0$ since forms are
dual to currents.

Suppose $d\beta$ is a real exact $(1,1)-$form, then for any $v\in
\mathcal Z$, we have
$$d\beta\cdot \pi_{1,1}v=d\beta \cdot v =\beta \cdot dv=0,$$
i.e. $d\beta$ vanishes on $\pi_{1,1}\mathcal Z$. For the converse
statement, suppose $\gamma$ is a real $(1,1)-$form vanishing on
$\pi_{1,1}\mathcal Z$. Then $\gamma$ vanishes on $\mathcal Z$. By
Theorem 17' in \cite{DR}, $\gamma$ is exact.
\end{proof}


Since a real $(1,1)-$form vanishes on  $\pi_{1,1}\mathcal B$
($\pi_{1,1}\mathcal Z$) if and only if it vanishes on
$\overline{\pi_{1,1}\mathcal B}$ ($\overline{\pi_{1,1}\mathcal Z}$),
we have

\begin{lemma}\label{closed-exact-bar}
A real $(1,1)-$form is closed if and only if it vanishes on
$\overline{\pi_{1,1}\mathcal B}$, and a real $(1,1)-$form is exact
if and only if it vanishes on $\overline{\pi_{1,1}\mathcal Z}$. The
same is true in the $(2,0), (0,2)$ setting.
\end{lemma}

The following is a crucial observation based on Lemma
\ref{closed-exact-bar} and the Hahn-Banach Theorem.
\begin{prop}\label{important'}
For any almost complex structure $J$, there is a natural isomorphism
\begin{equation}\bar \sigma^{1,1}: H^{1,1}_J(M)_{\mathbb R}=\frac{\mathbf
Z^{1,1}}{\mathbf B^{1,1}}\to (\frac{\overline{\pi_{1,1}\mathcal
Z}}{\overline{\pi_{1,1}\mathcal B}})^*.
\end{equation}
Similarly, there are natural isomorphisms $\bar \sigma^{(2,0),
(0,2)}, \bar\sigma_{1,1}, \bar \sigma_{(2,0), (0,2)}$.
\end{prop}

\begin{proof}
By Lemma  \ref{closed-exact-bar},  closed $(1,1)-$forms vanish on
$\overline{\pi_{1,1}\mathcal B}$, thus there is a homomorphism
$\sigma':\mathbf Z^{1,1}\to (\frac{\overline{\pi_{1,1}\mathcal
Z}}{\overline{\pi_{1,1}\mathcal B}})^*$. The kernel of $\sigma'$ is
$\mathbf B^{1,1}$ by Lemma \ref{closed-exact}. Therefore
$\bar\sigma^{1,1}$ is well-defined and an injection.

The surjectivity is proved by an application of the Hahn-Banach
Theorem to construct a homomorphism in the reverse direction. An
element in $(\frac{\overline{\pi_{1,1}\mathcal
Z}}{\overline{\pi_{1,1}\mathcal B}})^*$ is the same as a functional
$L$ on $\overline{\pi_{1,1}\mathcal Z}$ vanishing on
$\overline{\pi_{1,1}\mathcal B}$. Since $\overline{\pi_{1,1}\mathcal
Z}$ is closed subspace of $\mathcal E_{1,1}(M)_{\mathbb R}$,  we can
extend $L$ to a functional $\mathcal L$ on $\mathcal
E_{1,1}(M)_{\mathbb R}$ vanishing on $\overline{\pi_{1,1}\mathcal
B}$. Such a $\mathcal L$ corresponds to a closed $(1,1)-$form by
Lemma \ref{closed-exact}. If $\mathcal L'$ is another extension of
$L$, then $\mathcal L'-\mathcal L$ vanishes on
$\overline{\pi_{1,1}\mathcal Z}$ by Lemma  \ref{closed-exact}. And
since We have shown that there is a homomorphism
$\tau:(\frac{\overline{\pi_{1,1}\mathcal
Z}}{\overline{\pi_{1,1}\mathcal B}})^*\to H^{1,1}_J(M)_{\mathbb R}.$
Since $\bar\sigma^{1,1} \tau=\rm{id}$, $\bar\sigma^{1,1}$ is also
surjective.
\end{proof}

If $J$ is closed, by Lemma \ref{sigma}, $\bar\sigma^{1,1}$ becomes
the  isomorphism $$\sigma^{1,1}: H^{1,1}_J(M)_{\mathbb
R}=\frac{\mathbf Z^{1,1}}{\mathbf B^{1,1}}\to
(\frac{\pi_{1,1}\mathcal Z}{\pi_{1,1}\mathcal B})^*.
$$
Together with Lemma \ref{injective}, we obtain
\begin{cor}\label{dim} Let $J$ be a closed and  pure almost complex
structure.  Then we have \begin{equation}\label{geq}\dim
H_J^{1,1}(M)_{\mathbb R}\geq \dim H^J_{1,1}(M)_{\mathbb R}.
\end{equation}  If $J$ is also full, then $H^J_{1,1}(M)_{\mathbb R}$ and
$H_J^{1,1}(M)_{\mathbb R}$ have the same dimension and are dual to
each other.
\end{cor}

By Lemmas \ref {schaefer}, \ref{also},  and Proposition
\ref{important'}, we have

\begin{cor}\label{full}
Suppose  $\iota^{1,1}\,\, (\iota^{(2,0), (0,2)})$ is surjective,
 then \begin{equation}\label{geq(geq)}\dim H_J^{1,1}(M)_{\mathbb R}\geq \dim
H^J_{1,1}(M)_{\mathbb R}\,\,\,( \dim H_J^{(2,0), (0,2)}(M)_{\mathbb
R}\geq \dim H^J_{(2,0), (0,2)}(M)_{\mathbb R}).
\end{equation}
Suppose  $\iota_{1,1}\,\, (\iota_{(2,0), (0,2)})$ is surjective,
then the reverse inequality holds.
\end{cor}

By Lemmas \ref{iota'} and \ref{injective} and Corollary \ref{full},
we obtain
\begin{cor}
Suppose  $J$ is $C^{\infty}$ full,
 then \begin{equation}\label{geqgeq}\begin{array}{lll}\dim H_J^{1,1}(M)_{\mathbb R}&\geq &\dim
H^J_{1,1}(M)_{\mathbb R}, \cr \dim H_J^{(2,0), (0,2)}(M)_{\mathbb
R}&\geq &\dim H^J_{(2,0), (0,2)}(M)_{\mathbb R}.
\end{array}\end{equation}
Suppose $J$ is full, then the reverse inequalities  hold.
\end{cor}

Notice that   $H^J_{1,1}(M)_{\mathbb R}$ pairs trivially with
$H_J^{(2,0), (0,2)}(M)_{\mathbb R}$, and $H_J^{1,1}(M)_{\mathbb R}$
pairs trivially with $H^J_{(2,0), (0,2)}(M)_{\mathbb R}$.
 Suppose
$e\in H^J_{1,1}(M)_{\mathbb R}\cap H^J_{(2,0), (0,2)}(M)_{\mathbb
R}$. Then $e$ pairs trivially with $H_J^{1,1}(M)_{\mathbb R}+
H_J^{(2,0), (0,2)}(M)_{\mathbb R}$. Since the pairing of
$H_2(M;\mathbb R)$ and $H^2(M;\mathbb R)$ is non-degenerate, we have

\begin{prop}\label{simple}
If $J$ is $C^{\infty}$ full, then it  is pure and we have
\eqref{geqgeq}.
 Similarly, if $J$ is full, then it is
$C^{\infty}$ pure and the reverse inequalities hold.

 In
particular, if $J$ is both full and $C^{\infty}$ full, then it is
pure and full, as well as $C^{\infty}$ pure and full.

If $J$  is pure and full as well as $C^{\infty}$ pure and full, then
the $J-$homology and cohomology decompositions of $H_2(M;\mathbb R)$
and $H^2(M;\mathbb R)$ are dual to each other.
\end{prop}

\begin{remark} \label {ft} It is  observed in \cite{FT} that if $J$ is $C^{\infty}$ pure
and full, then it is pure. Fino and Tomassini  further show that if
$J$ is a $C^{\infty}$ pure and full almost complex structure on a
$2n-$dimensional manifold, then $J$ is also pure and full under any
of the following three conditions:
\begin{itemize}
\item $n=2$, or

\item there is a metric $g$ compatible with $J$ such that any
cohomology class in $H_J^{1,1}(M)_{\mathbb R}$ ($H_J^{(2,0),
(0,2)}(M)_{\mathbb R}$) has a $g-$harmonic representative in
$\mathcal Z_J^{1,1}$ ($\mathcal Z^{(2,0), (0,2)}$) respectively,  or

\item $J$ is compatible with a symplectic form $\omega$  of Lefschetz
type.
\end{itemize}

The first condition implies that when $n=2$ any $J$ is pure and full
(see Remark \ref{last}), while the first and the third conditions
imply Proposition \ref{complex3}. Many families of pure and full
(non-K\"ahler) almost complex structures on compact nilmanifolds and
solvmanifolds are also constructed this way in \cite{FT}.
\end{remark}

\section{Symplectic forms versus complex cycles}
\subsection{Complex cycles and tamed symplectic forms}
 In this
subsection we describe  Sullivan's approach to tamed symplectic
forms.
\subsubsection{Structure cycles and closed transversal forms}

Let us first review some general basic concepts in \cite{S}.
 A compact convex cone
$\mathcal C$ in a locally convex topological space over $\mathbb R$
is a convex cone which for some continuous linear functional $L$
satisfies $L(w)>0$ for $w\ne 0$ in $\mathcal C$ and $L^{-1}(1)\cap
\mathcal C$ is compact. The latter set is called the base for the
cone.

A cone structure (of $2-$directions) on $M$ is a continuous field of
compact convex cones $\{\mathcal C_x\}$ in the vector space
$\Lambda_2(x)=\Lambda^2T_xM$ of tangent $2-$vectors on $M$. Such a
cone structure is called ample if at each point $x$ the cone
$\mathcal C_x$ intersects the linear span of the Schubert variety
$S_{\tau}$ of every $2-$plane $\tau$ at $x$ ($S_{\tau}$ is the set
of $2-$planes which intersect $\tau$ in at least a line).

A smooth $2-$form $\omega$ on $M$ is \textit{transversal} to the cone
structure $\mathcal C$ if $\omega(v)>0$ for each $v\ne 0$ in
$\mathcal C_x\subset \Lambda_2(x), x\in M$. Using a partition of
unity it is easy to see that such transversal forms always exist.

A Dirac current is one determined by the evaluation of $2-$forms on
a single $2-$vector at one point. The cone of structure currents
$\mathfrak C$ associated to the cone structure $\mathcal C$ is the
closed convex cone of currents generated by the Dirac currents
associated to the elements of $\mathcal C_x, x\in M$. It is easy to
see that if $M$ is compact then $\mathfrak C$ is a compact convex
cone.

The structure cycles of $\mathcal C$ are the structure currents of
$\mathcal C$ which are closed as currents. Let $Z\mathfrak C$ be the
cone of structure cycles and  let $H\mathfrak C$ be the cone of
homology classes of structure cycles in $H_2(M;\mathbb R)$. Let
$\breve {H\mathfrak C}\subset H^2(M;\mathbb R)$ be the dual cone
defined by $(\breve {H\mathfrak C}, H\mathfrak C)\geq 0$. Notice
that the interior of $\breve {H\mathfrak C}$ is characterized by
$(\hbox{Int} \breve {H\mathfrak C}, H\mathfrak C)> 0$ when $H\mathfrak
C\neq \{0\}$. That is because if we have an element $a$ in the
interior of $\breve {H\mathfrak C}$, pair with some element $s \in
H\mathfrak C$ is 0, then there exists $a'$ such that $(a',s)<0$.
Thus $(a+ta',s)<0$ for small positive $t$. That is a contradiction.

Let $\mathcal E_p(M)_{\mathbb R}$ be the space of real
$p-$dimensional currents. Let $\mathcal B\subset \mathcal
E_2(M)_{\mathbb R}$ be the subspace of
  boundaries and $\mathcal Z\subset \mathcal
E_2(M)_{\mathbb R}$ be the subspace of
 cycles.

Sullivan made the following beautiful observation using the duality
between forms and currents, the Hahn-Banach theorem, and the
compactness.

\begin{theorem}(\cite{S})\label{ample} Let $\mathcal C$ be an ample cone
structure. Then $\mathcal C$ has non-trivial cycles. And if no
non-trivial  structure cycle is homologous to zero, i.e.
$$Z\mathfrak C\cap \mathcal B=\{0\},$$
 then $M$ admits a
closed $2-$form transverse to the ample $2-$direction structure.
Moreover, $H\mathfrak C$ is a compact convex cone and the interior
of $\breve {H\mathfrak C}\subset H^2(M;\mathbb R)$ consists
precisely of the classes of closed forms transverse to $\mathcal C$.
\end{theorem}

\subsubsection{Complex cycles} Let $J$ be an
 almost complex structure on $M$. Let $\mathcal
C(J)$ be the cone structure of complex lines. As for any plane
$\tau$ and any nonzero vector $v$ in $\tau$ the complex line $(v,
Jv)$ intersects $\tau$ at least in $v$, $\mathcal C(J)$ is an ample
cone structure.

 In particular, a $2-$dimensional current is of
bidimension $(1,1)$  if it can be locally expressed as
$$T=\sum_{j,k}T^{jk}\frac{\sqrt{-1}}{2}X_j\wedge \bar X_k,$$
where $T^{jk}$ is a distribution and $\{X_1, X_2, \dots\}$ is a
basis of type $(1,0)$ vectors. Such a current $T$ is said to be
positive if $\sum T^{jk}w_j\bar w_k$ is a non-negative measure for
each $w\in \mathbb C^n$. Using geometric measure theory it is
observed in \cite{HL} that

\begin{lemma}\label{positive}
A structure current associated to $\mathcal C(J)$ is a positive
current of bidimension $(1,1)$.
\end{lemma}
The argument is as follows.
 Firstly, there exists a
non-negative Radon measure $\|T\|$ called the total variation
measure of $T$ and a $2-$vector $\vec T$, which is $\|T\|$
measurable, such that $T=\|T\|\vec T$. Secondly, a $2-$dimensional
current $T$ is positive if and only if $\vec T(x)$ belongs to
$\mathcal C(J)_{x}$ for each $x\in M$.

In this case a structure cycle is called a complex cycle.



More generally, we can consider the cone of complex cycles of
$\mathbb C-$dimension $p$ and denote it by $C_p$. In particular,
$C_1=Z\mathfrak C(J)$. Let $D_p\subset  C_p$ denote the subcone of
diffuse complex cycles, which consists of  currents in $ C_p$ given
by closed $2n-2p$ forms.

The following fact was   noted in \cite{S}.

\begin{lemma}
A $J-$compatible form $\alpha$ gives rises to  a diffuse complex
cycle in $D_{n-1}$. Conversely, a diffuse cycle in $D_{n-1}$ whose
interior support is all of $M$ is a $J-$compatible form.
\end{lemma}

To illustrate why this is true we look at the case $n=2$.  In a
basis $\{z_1=x_1+\sqrt {-1}y_1,z_2=x_2+\sqrt {-1}y_2\}$ of $\mathbb
C^2$ the canonical form
$$dx_1\wedge dy_1+
dx_2\wedge dy_2$$ on $\mathbb R^4$ corresponds to the ray given by
the sum of the coordinate complex lines
$$\{\frac{\partial}{\partial
x_1}, \frac{\partial}{\partial y_1}=\sqrt{-1}
\frac{\partial}{\partial x_1}\}\quad \hbox{and} \quad
\{\frac{\partial}{\partial x_2}, \frac{\partial}{\partial
y_2}=\sqrt{-1} \frac{\partial}{\partial x_2}\}$$
 of $\mathbb
C^2$ in $\Lambda _2$.

 It was
also observed in \cite{S} that the  natural intersection pairing
between forms and currents satisfies $D_p\cdot C_q\subset
C_{p+q-n}$.
 In particular,  the diffuse complex cycles form a semi-ring under intersection.
 This implies that
if $\alpha_i, 1\leq i\leq n$ are $n$ $J-$compatible forms, then
$\prod_{i=1}^n\alpha_i>0$.
 This is true because, at each point $x$, the form $\wedge_{i=1}^n\alpha_i>0$ is
positive on a complex basis of the form $\{v_i, Jv_i\}$.

Finally, we write down the geometric part of $H\mathfrak C(J)$ for a
complex structure $J$.
 When $J$ is a complex structure,
a $p-$dimensional analytic subset of $M$ is in $\mathbb C_p$. Hence
 there are the following
geometric objects in $Z\mathfrak C(J)$: If $Y$ is a $p-$dimensional
analytic subset of $M$ and  $\omega$ is a K\"ahler form, then
$Y\wedge \omega^{p-1}\in Z\mathfrak C(J)$. Define
\begin{equation}\label{T}
T= \hbox{PD}((\mathcal K_J^c)^{n-1}) +HC\mathfrak Y^{1}\cdot
\hbox{PD}((\mathcal K_J^c)^{n-2})+ \cdots +HC\mathfrak Y^{n-1},
\end{equation}
 where $HC\mathfrak Y^{i}\subset
H_J^{i,i}(M)_{\mathbb R}$ is the cone generated by classes of
$(n-i)-$dimensional irreducible analytic subsets of $M$. $T\subset
H\mathfrak C(J)$ is considered to be the geometric part of
$H\mathfrak C(J)$.

We might be able to  define the analogue of $T$ for a general almost
complex structure replacing analytic subsets by either the zero sets
or the  images of locally pseudo-holomorphic maps.

\subsubsection{Tamed symplectic forms} \label{2.2}

 Notice
that a closed  form transverse to $\mathcal C(J)$ is nothing but a
$J-$tamed symplectic form. As an almost symplectic form on a closed
manifold cannot be exact, it was observed in \cite{S} that the cone
of complex cycles $Z\mathfrak C(J)$ is non-empty. Moreover,
$Z\mathfrak C(J)$ is a compact convex cone in the space of
$2-$currents.
As a consequence of Theorem \ref{ample} we have

\begin{theorem} (\cite{S}) \label{sullivan}
Let $(M, J)$ be an almost complex $2n-$manifold. Then  $\mathcal
K_J^t$ is non-empty if and only if there is no non-trivial positive
current of bidimension $(1,1)$ which is a boundary. Moreover, under
the (non-degenerate) pairing between  $H^2(M;\mathbb R)$ and
$H_2(M;\mathbb R)$, $\mathcal K_J^t\subset H^2(M;\mathbb R)$ is the
interior of the dual cone of $H\mathfrak C(J)\subset H_2(M;\mathbb
R)$.
\end{theorem}

\subsection{Complex cycles and compatible symplectic forms}
Notice that a $2-$form $\omega$ is a $J-$compatible symplectic form
if and only if it is of type $(1,1)$, closed and tamed by $J$.

A $(1,1)-$form corresponds to a functional on $\mathcal
E_{1,1}(M)_{\mathbb R}$. By Lemma \ref{closed-exact}  a functional
on $\mathcal E_{1,1}(M)_{\mathbb R}$ gives rise to a closed form if
and only if it vanishes on $\pi_{1,1}\mathcal B$. By definition, a
functional on $\mathcal E_{1,1}(M)_{\mathbb R}$ gives rises to a
$J-$tamed form if and only if it is positive on $\mathcal
C(J)\backslash \{0\}$. Thus we have

\begin{lemma}\label{bij}
There is a bijection between  $J-$compatible symplectic forms and
functionals on $\mathcal E_{1,1}(M)_{\mathbb R}$ vanishing on
$\overline{\pi_{1,1}\mathcal B}$ and positive on $\mathcal
C(J)\backslash \{0\}$.
\end{lemma}

When $J$ is integrable there is the following K\"ahler criterion in
\cite{HL}.

\begin{theorem} \label{HL} Suppose $(M,J)$ is a compact complex manifold.
Then there exists a K\"ahler metric for $M$ if and only if
\begin{equation}\label{kahler}
\mathfrak C(J)\cap \pi_{1,1}\mathcal B={0}.
\end{equation}
\end{theorem}

By Lemma \ref{bij}, for any  almost K\"ahler complex structure,
 \eqref{kahler} is replaced by
\begin{equation}\label{almost-kahler}
\mathfrak C(J)\cap \overline{\pi_{1,1}\mathcal B}={0}.
\end{equation}
  We can further characterize $\mathcal
K_J^c$.
\begin{theorem} \label{generalization} Suppose $(M, J)$ is an almost complex manifold.
 If $\mathcal K_J^c$ is non-empty
then
 $\mathcal K_J^c\subset H_J^{1,1}(M)_{\mathbb R}$ is the interior
of the dual cone of $H\mathfrak C(J)\subset H^J_{1,1}(M)_{\mathbb
R}$.
\end{theorem}

  \begin{proof}

First of all, under the (possibly degenerate) pairing between
$H_J^{1,1}(M)_{\mathbb R}$ and $H^J_{1,1}(M)_{\mathbb R}$, $\mathcal
K_J^c$ is contained in the interior of the dual cone of $H\mathfrak
C(J)$ by the definition of the positive current and lemma
\ref{positive}.

It remains to prove that if $e\in H_J^{1,1}(M)_{\mathbb R}$ is
positive on $H\mathfrak C(J)$, then it is represented by a
$J-$compatible form. Consider the element $\bar\sigma(e)$  in
$(\frac{\overline{\pi_{1,1}\mathcal Z}}{\overline{\pi_{1,1}\mathcal
B}})^*$, which pulls back to a functional $L$ on
$\overline{\pi_{1,1}\mathcal Z}$ vanishing on
$\overline{\pi_{1,1}\mathcal B}$. Denote the kernel hyperplane of
$L$ in $\overline{\pi_{1,1}\mathcal Z}$ also by $L$. By our choice
of $e$, as subsets of $\mathcal E_{1,1}(M)_{\mathbb R}$, $L$ and
$\mathfrak C(J)\backslash \{0\}$ are disjoint.

Since $\overline{\pi_{1,1}\mathcal Z}$ is a closed subspace of
$\mathcal E_{1,1}(M)_{\mathbb R}$, any kernel hyperplane in
$\overline{\pi_{1,1}\mathcal Z}$ containing
$\overline{\pi_{1,1}\mathcal B}$ is also a closed subspace of
$\mathcal E_{1,1}(M)_{\mathbb R}$.

Choose a Hermitian metric $h$ and let $\psi$ be the
 associated real $(1,1)$ form. Set
 $$\overline {\mathfrak C}(J)=\{T\in \mathfrak C(J)| T(\psi)=1\}.$$
 It was shown in \cite{S} that $\overline {\mathfrak C}(J)$ is
 compact in $\mathcal E_{1,1}(M)_{\mathbb R}$.

By a variation of the ``second separation theorem" (Schaefer
\cite{Sch} p65), we get a hyperplane $\mathcal L$ in
$\mathcal E_{1,1}(M)_{\mathbb R}$  containing $L$ and disjoint from
$\overline {\mathfrak C}(J)$. The hyperplane $\mathcal L$ determines
a functional $\alpha$ vanishing on $L$ and being positive on
$\overline {\mathfrak C}(J)$. By Lemma \ref{bij}, $\alpha$ is a
$J-$compatible symplectic form. Moreover, by construction we have
$[\alpha]=e$.
\end{proof}

Since $\mathcal K_J^c$ is open in $H_J^{1,1}(M)_{\mathbb R}$, via
the pairing between $H_J^{1,1}(M)_{\mathbb R}$ and
$H^J_{1,1}(M)_{\mathbb R}$,  $H_J^{1,1}(M)_{\mathbb R}$ injects into
the dual space of $H^J_{1,1}(M)_{\mathbb R}$. Thus we have
\begin{cor} Suppose $(M, J)$ is an almost complex manifold with $J$ almost
K\"ahler, i.e.  $\mathcal K_J^c\ne \emptyset$. Then $$\dim
H_J^{1,1}(M)_{\mathbb R}\leq \dim H^J_{1,1}(M)_{\mathbb R}.$$
\end{cor}

The natural questions are whether $H\mathfrak C(J)$ is also dual to
$\mathcal K_J^c$. If it is true, then when $H\mathfrak C(J)$ is open
in $H^J_{1,1}(M)_{\mathbb R}$, we have the equality. This is the
case for a K\"ahler structure.

 \subsection{Comparing  $\mathcal K_J^c$ and $ \mathcal K_J^t$}

 We start with

\begin{prop}\label{general} Suppose $(M, J)$ is an almost complex manifold with $J$ almost
K\"ahler, i.e.  $\mathcal K_J^c\ne \emptyset$. Then  we have
\begin{equation}\label{ct'}
\mathcal K^t_J \cap H_J^{1,1}(M)_{\mathbb R}=\mathcal K^c_J,
\end{equation}
and
\begin{equation}\label{sup'}\mathcal
K_J^c+H_J^{(2,0),(0,2)}(M)_{\mathbb R}\subset \mathcal K_J^t.
\end{equation}
\end{prop}

\begin{proof}
\eqref{ct'} is a direct consequence of Theorem \ref{generalization}
and Theorem \ref{sullivan}.

 For the inclusion \eqref{sup'} we can actually explicitly find
$J-$tamed symplectic forms in
$$\mathcal
K_J^c+H_J^{(2,0),(0,2)}(M)_{\mathbb R}.$$
 This is
discussed in a slightly different context in \cite{Pe} and
\cite{Dr}. The point is,  $\beta\in \Omega_J^{(2,0),
(0,2)}(M)_{\mathbb R}$ vanishes on any complex line, i.e. at any
point $x\in M$, $\beta(v, Jv)=0$ for $v\in T_xM$. Thus for
 an almost K\"ahler form  $\alpha\in \Omega_J^{1,1}(M)_{\mathbb R}$,
 the form
$\alpha+t\beta$, restricted to a complex line,  is equal to $\alpha$
and hence is positive.  If $\beta$ is also closed, then, as already
observed in \cite{Dr2}, $\alpha+t\beta$ is a $J-$tamed symplectic
form for any $t\in \mathbb R$.

\end{proof}

As mentioned in Remark \ref {last},  such a $J$ is actually
$C^{\infty}$ pure.

We are ready to prove Theorem \ref{extension}.
\begin{proof}
By Proposition \ref{general}, it suffices to prove the reverse
inclusion
\begin{equation}\label{sub'}\mathcal
K_J^c+H_J^{(2,0),(0,2)}(M)_{\mathbb R}\supset \mathcal K_J^t.
\end{equation}
Since  $J$ is assumed to be $C^{\infty}$ full, by Theorem
\ref{generalization}, we  have
$$\hbox{Int } \breve{H\mathfrak C(J)} \subset \mathcal
K_J^c+H_J^{(2,0),(0,2)}(M)_{\mathbb R}.$$  In addition, by Theorem
\ref{sullivan}, $\hbox{Int } \breve{H\mathfrak C(J)}=
 \mathcal K_J^t$. Thus we have the inclusion \eqref{sub'}.
\end{proof}

As a consequence of Theorem \ref{extension} and Proposition
\ref{complex1}  we have

\begin{cor} \label{first} Suppose J is a complex  structure on $M$. If
 $\mathcal K_J^c\ne \emptyset$, i.e. $J$ is K\"ahler,  then
\begin{equation}\label{tc}
\mathcal K_J^t= \mathcal K_J^c+(H_{\bar\partial}^{2,0}(M)\oplus
H_{\bar\partial}^{0,2}(M))_{\mathbb R}.
\end{equation}
In particular, if $\mathcal K_J^c\ne \emptyset$, then
\begin{equation}\label{ct}
\mathcal K_J^t\cap H_{\bar\partial}^{1,1}(M)_{\mathbb{R}}=\mathcal
K_J^c.
\end{equation}
\end{cor}

\section{Complex structures}

In this section we assume that the almost complex structure $J$ is
integrable. For such a $J$  we say that it is K\"ahler if $\mathcal
K_J^c$ is non-empty.
\subsection{K\"ahler complex structures}

\subsubsection{Tamed cones and compatible cones}

 If $J$ is K\"ahler we have the following type decomposition,
\begin{equation}\label{type1}
\begin{array}{ll}
H^2(M;\mathbb C)&=H_{\bar\partial}^{1,1}(M)\oplus
H_{\bar\partial}^{2,0}(M)\oplus H_{\bar\partial}^{0,2}(M),\cr
H^2(M;\mathbb R)&=H_{\bar\partial}^{1,1}(M)_{\mathbb R}\oplus
(H_{\bar\partial}^{2,0}(M)\oplus H_{\bar\partial}^{0,2}(M))_{\mathbb
R},
\end{array}
\end{equation}
where $H_{\bar\partial}^{p,q}$ denotes the $(p,q)-$Dolbeault
cohomology group.
 Notice that, $H\mathfrak C(J)$ vanishes on $H_{\bar\partial}^{2,0}(M)\oplus
H_{\bar\partial}^{0,2}(M)$ and thus can be considered to be a
subcone of $H^{n-1,n-1}_{\bar\partial}(M)_{\mathbb R}$. And by the
$J-$invariance,
$$\mathcal
K_J^c \subset H_{\bar\partial}^{1,1}(M)_{\mathbb R}.$$ Moreover, we
have the following beautiful description of $\mathcal K_J^c$.

\begin{theorem} \label{DP} (\cite{DP}) Let $J$ be a K\"ahler complex structure on
a real $2n-$dimensional manifold $M$. Let $T \subset H\mathfrak
C(J)\subset H^{n-1,n-1}_{\bar \partial}(M)_{\mathbb R}$ be defined
by (\ref{T}). Then the interior of the dual cone of $T$ in
$H_{\bar\partial}^{1,1}(M)_{\mathbb R}$ equals $\mathcal K_J^c$.
\end{theorem}

An immediate consequence is

\begin{cor} \label{cyclecone} Let $J$ be a K\"ahler complex structure on
a real $2n-$dimensional manifold $M$. Then
\begin{equation}\label{cycle}
\begin{array}{ll}
&\hbox{Int }H\mathfrak C(J)=T\cr =&\hbox{PD}((\mathcal
K_J^c)^{n-1}) +HC\mathfrak Y^{1}\cdot \hbox{PD}((\mathcal
K_J^c)^{n-2})+ \cdots +HC\mathfrak Y^{n-1},
\end{array}
\end{equation}
 where
$HC\mathfrak Y^{i}\subset H_{\bar\partial}^{i,i}(M)_{\mathbb R}$ is
generated by classes of $(n-i)-$dimensional irreducible analytic
subsets of $M$.
\end{cor}

Here is another argument for Corollary \ref{first}.

\begin{proof}We first show that
\begin{equation}\label{sup}
\mathcal K_J^t\supset \mathcal
K_J^c+(H_{\bar\partial}^{2,0}(M)\oplus
H_{\bar\partial}^{0,2}(M))_{\mathbb R}.
\end{equation}
One proof of (\ref{sup}) is to evoke Theorem \ref{sullivan}.  We
only need to  observe that $(H_{\bar\partial}^{2,0}(M)\oplus
H_{\bar\partial}^{0,2}(M))_{\mathbb R}$ vanishes on $ H\mathfrak
C(J)$ and $\mathcal K_J^c$ is positive on $ H\mathfrak C(J)$. Thus
$\mathcal K_J^c+(H_{\bar\partial}^{2,0}(M)\oplus
H_{\bar\partial}^{0,2}(M))_{\mathbb R}$ is in the interior of the
dual of $ H\mathfrak C(J)$.

To prove the other direction, we apply Theorem \ref{DP} to obtain
\begin{equation}\label{sub}
\hbox{Int } \breve{H\mathfrak C(J)}= \hbox{Int } \breve T \subset
 \mathcal K_J^c+(H_{\bar\partial}^{2,0}(M)\oplus
H_{\bar\partial}^{0,2}(M))_{\mathbb R}.
\end{equation}

\end{proof}

To state the next result let us introduce a definition.

 \begin{definition} \label{np} A
degree 2 real cohomology class $V$ of a complex manifold $M$ is said
to be numerically positive on analytic cycles, or numerically
positive for short,   if $V^p$ pairs positively with  the homology
class of $Y$ for any irreducible analytic set $Y$ in $M$ with
$\dim_{\mathbb C} Y=p$. Let $\mathcal{NP}$ denote the set of
numerically positive classes.
\end{definition}

We have the following characterization of the $J-$tamed cone in
terms of numerically positive classes.

\begin{theorem}\label{JC}
Let $(M,J)$ be a compact K\"ahler manifold. Then the J-tamed cone
$\mathcal{K}_J^t$ is one of the connected components of
$\mathcal{NP}$.
\end{theorem}

\begin{proof}
Since the space of $J-$tamed symplectic forms is a convex set the
cone $\mathcal K_J^t$ is connected. Thus it suffices to show the
following two inclusions: one component of $ \mathcal{NP}\subset
\mathcal{K}_J^t$ and $\mathcal{K}_J^t \subset \mathcal{NP}$.

Let $\mathcal{NP}^{1,1}$ be the $(1,1)$ part of $\mathcal {NP}$.
Then
 by the main result of \cite{DP} (c.f. Theorem
\ref{KC}) one component of $\mathcal{NP}^{1,1}$ is contained in
$\mathcal K^c_J$. Denote this component by
$\widetilde{\mathcal{NP}^{1,1}}$. Then by (\ref{tc}) we have
$$\widetilde{\mathcal{NP}^{1,1}}+
(H_{\bar\partial}^{2,0}(M)\oplus H_{\bar\partial}^{0,2}(M))_{\mathbb
R}\subset \mathcal K_J^c+(H_{\bar\partial}^{2,0}(M)\oplus
H_{\bar\partial}^{0,2}(M))_{\mathbb R}=\mathcal K_J^t.$$ On the
other hand, the component of $\mathcal{NP}$ which contains
$\widetilde{\mathcal{NP}^{1,1}}$ definitely belongs to
$$\widetilde{\mathcal{NP}^{1,1}}+ (H_{\bar\partial}^{2,0}(M)\oplus
H_{\bar\partial}^{0,2}(M))_{\mathbb R}.$$  Thus we have proved the
first inclusion.

To prove $\mathcal{K}_J^t \subset \mathcal{NP}$, write any $V\in
\mathcal{K}_J^t$ as $U+W$, where $U$ is the $(1,1)$ part. By
(\ref{tc}),  $U=[\alpha]$ for a K\"ahler form $\alpha$  and
$W=[\beta]$ for a real closed form in $ \Omega_J^{2,0}(M)\oplus
\Omega_J^{0,2}(M)$.
 Then $\alpha+\beta$ is a tamed symplectic form when restricted to the smooth part of any
 irreducible analytic set $Y$. It follows that $V$ is numerically
 positive, i.e. $\mathcal{K}_J^t \subset \mathcal{NP}$.

\end{proof}
\subsubsection{Relating to some classical theorems}

Corollary \ref{first} and Theorem  \ref{JC} are parallel to some
classical theorems in K\"ahler geometry and algebraic geometry.

Let us recall some basic concepts in algebraic geometry. For a
complex manifold $(M, J)$ a very ample line bundle $L$ is a
holomorphic line bundle on $M$ with enough holomorphic sections to
set up an embedding of  $M$ into a projective space. An ample line
bundle $L$ is one whose certain tensor power becomes very ample.
Taking  the first Chern class of every ample line bundle we obtain
 a cone $\mathcal{A}$, called the
\emph{ample cone}. Clearly $\mathcal A$ lies in  the integral
$(1,1)$ cohomology group, i.e.
 $$\mathcal{A}\subset
H^{1,1}(M; \mathbb Z):=H^2(M; \mathbb Z)\cap H^{1,1}(M; \mathbb
R).$$ In this subsection we will simply write $H^{1,1}(M;\mathbb R)$
for $H_{\bar\partial}^{1,1}(M)_{\mathbb R}$.

Recall we denote  the K\"ahler cone by $\mathcal K_J^c \subset
H^{1,1}(M; \mathbb R)$. The classical Kodaira embedding theorem can
then be stated in the following way as a comparison of the triples
$$\hbox{(projective, $H^{1,1}(M; \mathbb Z)$, $\mathcal{A}$) \quad and \quad (K\"ahler,
$H^{1,1}(M; \mathbb R)$, $\mathcal K_J^c$)}.$$

\begin{theorem} \label{Kodaira}
Let $(M,J)$ be a K\"ahler manifold. If $\mathcal{A} \neq \emptyset$,
then
\begin{equation}\label{ample3}
\mathcal{A}=\mathcal K_J^c \cap H^{1,1}(M; \mathbb Z).
\end{equation}
\end{theorem}

With this interpretation of the Kodaira embedding theorem our
Proposition \ref{first} can then be viewed as a (K\"ahler,
$J$-symplectic) analogue comparing the triples
$$\hbox{(K\"ahler,
$H^{1,1}(M; \mathbb R)$, $\mathcal K_J^c$) \quad  and \quad
($J$-symplectic, $H^2(M; \mathbb R)$, $\mathcal K_J^t$)}. $$ Here, a
$J$-symplectic manifold means a manifold with a symplectic form
tamed by a complex structure $J$. To see that let  $(M, J)$ be a
complex manifold with non-empty tamed cone $\mathcal K_J^t$. Then
the equation (\ref{ample3})  is exactly the analogue of the equation
(\ref{ct}).

In addition, our Theorem \ref{JC} is a Nakai-Moishezon type theorem
in the $J$-symplectic world. Recall that the original
Nakai-Moishezon characterizes the ample cone in terms of numerically
positive classes in Definition \ref{np}.
\begin{theorem}\label{NM}
The ample cone $\mathcal A$ of a projective manifold $(M, J)$ is
given by
\begin{equation}\label{ample2}
\mathcal{A}=\mathcal {NP} \cap H^{1,1}(M; \mathbb Z).
\end{equation}
\end{theorem}

Recall also the recent remarkable extension of Demailly-Paun in
\cite{DP} (first established in \cite{B} and \cite{L} in the case of
$n=2$).

\begin{theorem}\label{KC}
 The K\"ahler cone
of a K\"ahler manifold $(M, J)$  is one of the connected components
of
\begin{equation}\label{kahler2} \mathcal {NP} \cap
H^{1,1}(M; \mathbb R).
\end{equation}
\end{theorem}

Theorem \ref{JC} can be trivially restated as that the tamed cone of
a K\"ahler manifold $(M, J)$ is one of the connected components of
\begin{equation}\label{tame2} \mathcal {NP} \cap
H^{2}(M; \mathbb R).
\end{equation}

The parallel between (\ref{ample2}), (\ref{kahler2}) and
(\ref{tame2}) is clear.
 Especially, in dimension 4, by virtue of
the second statement of Theorem \ref{main}, we can replace the first
sentence in Proposition \ref{JC}, ``Let $(M,J)$ be a compact
K\"ahler surface",  by ``Let $(M,J)$ be a compact J-symplectic
surface". Thus, at least in dimension 4, replacing the triple
$$\hbox{($J$-symplectic, $H^2(M; \mathbb R)$, $\mathcal K_J^t$)}$$
 by the
triple
$$\hbox{(projective, $H^{1,1}(M; \mathbb Z)$,
$\mathcal{A}$)}$$
 specializes to  the classical Nakai-Moishezon
theorem.

\subsection{Complex surfaces}
In this subsection we specialize to complex dimension $2$.
\subsubsection{Tamed cones and compatible cones}
 We start with the proof of
Theorem \ref{main}.

\begin{proof} If $\mathcal K_J^t(M)$ is empty, then clearly $\mathcal
K_J^c(M)$ is empty as well.

Suppose $\mathcal K_J^c(M)$ is empty. Then it follows from \cite{S},
\cite{T}, \cite{Mi}, \cite{B} that $b_1(M)$ is odd.
 If $b^+(M)=0$ then $M$ has no symplectic structures, hence
$\mathcal K_J^t(M)$ is empty. The remaining case is $b_1$ odd and
$b^+\geq 1$. In this case,
 by the Kodaira classification of surfaces (see \cite{BPV}),  $(M,J)$ is elliptic. By
\cite{HL}, for an elliptic surface $(M, J)$ with $b_1$ odd,  the
torus fibers bound, thus there cannot be tamed symplectic forms.
\end{proof}

\begin{example}
 The Kodaira-Thurston manifold  is an  example with $b_1$ odd. Thus $\mathcal
K_J^t$ is empty for every {\it integrable} $J$. But it does admit
symplectic structures hence $\mathcal K_J^t$ is not empty for some
non-integrable $J$.
\end{example}

It follows from Theorem \ref{sullivan} and Corollary \ref{cyclecone}
that, for a complex surface $(M, J)$,  $H\mathfrak C(J)$ contains
the 0 element if $J$ is non-K\"ahler, and is equal to
 $PD (\mathcal K_J^c)+HC\mathfrak Y$ if $J$ is K\"ahler.

\begin{remark} \label{buchdahl} In this case a slightly different
 description of $\mathcal K_J^c$  was first proved
in   \cite{B}, \cite{L} by Buchdahl and Lamari independently:
  Let $(M, J)$ be a complex surface. If $\mathcal K_J^c$ is non-empty,
then $\mathcal K_J^c$ is a connected component of
$$\{U\in
H_{\bar\partial}^{1,1}(M)_{\mathbb R}|\hbox{$U^2>0$ and $U$ is
positive on $HC\mathfrak Y$}\},$$ where $HC\mathfrak Y$ is the
convex cone in $H^2(M;\mathbb R)$ generated by classes of
 irreducible holomorphic curves.

Let $C^-\mathfrak Y$ denote the subcone generated by irreducible
curves of negative square. It follows from the Hodge index theorem
that the K\"ahler cone $K_J^c$ is given by a connected component of
$$\{U\in H_{\bar\partial}^{1,1}(M)_{\mathbb R}|\hbox{$U^2>0$ and $U$ is positive
on $HC^-\mathfrak Y$}\}.$$ In particular, if there are no
irreducible curves of negative square, the K\"ahler cone coincides
with  one connected component of the $(1,1)$ positive cone.
 Such examples include generic complex structures
on $K3$ or $T^4$.
\end{remark}

In higher dimensions we speculate that, for an integrable $J$, it is
possible that $K_J^c$ is empty while $K_J^t$ is nonempty. Such
examples could come from the total space of a holomorphic bundle
over a K\"ahler manifold admitting a class whose restriction is in
the tamed cone of each fiber. In such a situation, on the one hand,
there is the construction of tamed symplectic forms in \cite{G}, and
on the other hand, there are new subtle obstructions being K\"ahler
in \cite{V}. Here is a potential example inspired by \cite{V}.

\begin{example} \label{Voisin}
Let $T^{10}$ be the real torus of dimension 10, and let $K$ be  any
simply connected compact K\"ahler manifold satisfying the conditions
that {\rm rank}$ H^2(K;\mathbb{Q})=11$ and that the cohomology of
$K$ is generated in degree 2.  By K\"unneth decomposition and
Poincar\'e duality we have an inclusion ${\rm
Hom}(H^2(T^{10};\mathbb Q), H^2(K;\mathbb Q))\subset
H^{10}(T^{10}\times K;\mathbb Q)$. Consider a generic surjective map
$\mu:H^2(T^{10};\mathbb Q)\to H^2(K;\mathbb Q)$ and let $\lambda\in
H^{10}(T^{10}\times K;\mathbb Q)$ be the image of $\mu$ under the
inclusion.

 Suppose there is a
holomorphic bundle $E$ over $Y:=T^{10} \times K$ such that
$$c_1(E)=0,\quad \hbox{and} \quad c_5(E)=m\lambda$$ for some non zero integer
$m$.
Consider the complex manifold  $\mathbb{P}(E)$. By Theorem 3.5  in
\cite{V}, the complex manifold $\mathbb P(E)$ is not K\"ahler.

On the other hand, since the first Chern class of the tautological
line bundle of $\mathbb{P}(E)$  restrict to  a symplectic class of
each fiber, the complex manifold $\mathbb{P}(E)$ admits a tamed
symplectic structure by \cite{G}.

\end{example}

\subsubsection{Integrable tamed cone versus integrable compatible
cone} In this subsection let us now change our perspective slightly.
Rather than fixing a pair $(M, J)$ we fix a manifold $M$  and
consider all possible complex structures $J$  on $M$.

\begin{definition} Let $\mathcal K_{Int}^t(M)$ be the union of
${\mathcal K}_J^t(M)$ over all integrable $J$ on $M$,  called the
integrable tamed cone of $M$. Similarly define the integrable
compatible cone ${\mathcal K}_{Int}^c(M)$.
\end{definition}

$\mathcal K_{Int}^c(M)$ is called the cohomology K\"ahler cone in
\cite{Pe} and \cite{Dr}.
 In general
the integrable tamed cone is not equal to  the integral compatible
cone. For example, for a ball quotient $(M,J_0)$, the complex
structure $J_0$ is the unique  complex structure, thus by the
equation (\ref{tc})
$${\mathcal K}_{Int}^c(M)=\mathcal K_{J_0}^c(M)\subsetneq \mathcal
K_{J_0}^t(M)=\mathcal K_{Int}^t(M),$$ if
$H_{\bar\partial}^{2,0}(M)\ne 0$.

We can also fix a degree 2 class $e$ in $H^2(M;\mathbb Z)$ and
consider the subcone $\mathcal K_{Int}^t(M, e)$ of $\mathcal
K_{Int}^t(M)$, called the $e-$integrable tamed cone of $M$. It is
the union of ${\mathcal K}_J^t(M)$ over all integrable $J$ on $M$
with $e$ being the first Chern class. Similarly we have the subcone
$\mathcal K_{Int}^c(M, e)$ of $\mathcal K_{Int}^c(M)$, the
$e-$integrable compatible cone of $M$.

In the case that $e$ is realized by $-\omega$ for some K\"ahler form
$\omega$ on $M$, by the Hodge-Riemann bilinear relations, it can be
shown that
\begin{equation}\label{not equal}  {\mathcal K}_{Int}^c(M,e)\subsetneq \mathcal K_{Int}^t(M,e),
\end{equation}
in certain situations. More precisely, let $J$ be a complex
structure for which $-e=[\omega]$ where $\omega$ is a K\"ahler form
for $J$. Suppose
 $\beta$ is a nonzero real and closed form in $\Omega_J^{2,0}(M)\oplus
\Omega_J^{0,2}(M)$. Results in  \cite{Dr2}, \cite{Dr} and \cite{Pe}
can be interpreted as saying  that  $-e+t[\beta]$ is not in
$\mathcal K_{Int}^c(M, e)$ if
\begin{itemize}
\item $n$ is even, $\beta^{\frac{n}{2}}$ is not identically zero, and
$|t|$ is sufficiently large, or

\item $n=2$, $t\ne 0$.

\end{itemize}
 We further observe that the same is true for $n$ odd and $|t|\ne 0$
sufficiently small. On the other hand, as already used in the proof
of Lemma \ref{general}, $\omega+t\beta$ is tamed by $J$ for any $t$.
Thus in these cases we have the strict inclusion (\ref{not equal}).

Coming back to the cones ${\mathcal K}_{Int}^c(M)$ and $\mathcal
K_{Int}^t(M)$, we have the following conjecture in the case of
$n=2$.

\begin{conj}  For any
 $M$ underlying a complex surface of general type
with $p_g\geq 1$, we have \begin{equation}\label{not equal2}
{\mathcal K}_{Int}^c(M)\subsetneq \mathcal K_{Int}^t(M),
\end{equation}
  and for all other $4-$dimensional manifolds, we
 have the the equality
\begin{equation}\label{equal}  {\mathcal K}_{Int}^c(M)=\mathcal K_{Int}^t(M).
\end{equation}
\end{conj}

This conjecture is plausible. In the case $n=2$, for a smooth
manifold $M$ underlying a minimal surface of general type, by the
Donaldson theory or the Seiberg-Witten theory (c.f. \cite{FQ},
\cite{Br}, \cite{FM}), the set $\{c_1(M, J), -c_1(M, J)\}$ is the
same for any $J$. Consequently, we have the strict inclusion
(\ref{not equal2}) if $b^+(M)\geq 3$ and there is a complex
structure on $M$ together with a K\"ahler form on $(M, J)$
representing its canonical class.

  Secondly,  an immediate consequence of Proposition \ref{main} is
that if $M$ has dimension 4 then  $\mathcal K_{Int}^t(M)$ is empty
if and only if $\mathcal K_{Int}^c(M)$ is empty.

 Thirdly, if  $p_g=0$ then $H^2(M;\mathbb
R)=H_{\bar\partial}^{1,1}(M)_{\mathbb R}$ and hence $\mathcal
K_J^t=\mathcal K_J^c$ by (\ref{tc}). Therefore we have (\ref{equal})
in this case.

Finally, when $p_g>0$, the equality (\ref{equal}) still holds when
the Kodaira dimension is zero (see \cite{Li}). We speculate that it
is also true when the Kodaira dimension is one, in particular, for
the  Elliptic surfaces $E(n)$.

\end{document}